\documentclass[11pt]{amsart}

\newtheorem{theorem}{Theorem}[section]

\newtheorem{corollary}[theorem]{Corollary}

\newtheorem{lemma}[theorem]{Lemma}

\newtheorem{problem}[theorem]{Problem}
\newtheorem{proposition}[theorem]{Proposition}

\theoremstyle{definition}
\newtheorem{remark}[theorem]{Remark}

\usepackage[colorlinks, bookmarks=true]{hyperref}
\usepackage{color,graphicx,shortvrb}

\usepackage{enumerate}
\usepackage{amsmath}
\usepackage{amssymb}

\usepackage[latin 1]{inputenc}

\def\J#1#2#3{ \left\{ #1,#2,#3 \right\} }
\def\RR{{\mathbb{R}}}

\def\11{\textbf{$1$}}
\def\CC{{\mathbb{C}}}

\begin{document}

\numberwithin{equation}{section}

\title[Approximation and convex decomposition]{Approximation and convex decomposition by extremals and the $\lambda$-function in JBW$^*$-triples}

\author[F.B. Jamjoom]{Fatmah B. Jamjoom}
\address{Department of Mathematics, College of Science, King Saud University, P.O.Box 2455-5, Riyadh-11451, Kingdom of Saudi Arabia.}
\email{fjamjoom@ksu.edu.sa}

\author[A.M. Peralta]{Antonio M. Peralta}
\address{Departamento de An{\'a}lisis Matem{\'a}tico, Universidad de Granada,\\
Facultad de Ciencias 18071, Granada, Spain}
\curraddr{Visiting Professor at Department of Mathematics, College of Science, King Saud University, P.O.Box 2455-5, Riyadh-11451, Kingdom of Saudi Arabia.}
\email{aperalta@ugr.es}

\author[A.A. Siddiqui]{Akhlaq A. Siddiqui}
\address{Department of Mathematics, College of Science, King Saud University, P.O.Box 2455-5, Riyadh-11451, Kingdom of Saudi Arabia.}
\email{asiddiqui@ksu.edu.sa}

\author[H.M. Tahlawi]{Haifa M. Tahlawi}
\address{Department of Mathematics, College of Science, King Saud University, P.O.Box 2455-5, Riyadh-11451, Kingdom of Saudi Arabia.}
\email{htahlawi@ksu.edu.sa}

\thanks{The authors extend their appreciation to the Deanship of Scientific Research at King Saud University for funding this work through research group no  RGP-361. Second author also partially supported by the Spanish Ministry of Science and Innovation, D.G.I. project no. MTM2011-23843.}

\begin{abstract} We establish new estimates to compute the $\lambda$-function of Aron and Lohman on the unit ball of a JB$^*$-triple. It is established that for every Brown-Pedersen quasi-invertible element $a$ in a JB$^*$-triple $E$ we have $$\hbox{dist} (a, \mathfrak{E} (E_1)) = \max \left\{ 1- m_q (a) , \|a\|-1\right\},$$ where $\mathfrak{E} (E_1)$ denotes the set of extreme points of the closed unit ball $E_1$ of $E$. It is proved that $\lambda (a) = \frac{1+m_q (a)}{2},$ for every Brown-Pedersen quasi-invertible element $a$ in $E_1$, where $m_q (a)$ is the square root of the quadratic conorm of $a$. For an element $a$ in $E_1$ which is not Brown-Pedersen quasi-invertible we can only estimate that $\lambda (a)\leq \frac12 (1-\alpha_q (a)).$ A complete description of the $\lambda$-function on the closed unit ball of every JBW$^*$-triple is also provided, and as a consequence, we prove that every JBW$^*$-triple satisfies the uniform $\lambda$-property.
\end{abstract}

\date{}
\maketitle

\section{Introduction}\label{sec:intro}

In \cite{AronLohm}, Aron and Lohman, defined a function on the closed unit ball, $X_1$, of an arbitrary Banach space $X$, which is determined by the geometric structure of the set $\mathfrak{E} (X_1)$ of extreme points of the closed unit ball of $X$. The mentioned function is called the \emph{$\lambda$-function} of the space $X$. The concrete definition reads as follows: Let us assume that $\mathfrak{E} (X_1)\neq \emptyset$, let $x$ and $y$ be elements in $X_1$, and let $e$ be an element in $\mathfrak{E} (X_1).$ For each $0<\lambda \leq 1,$ the ordered triplet $(e, y, \lambda)$ is said to be \emph{amenable} to $x$ when $x= \lambda e + (1-\lambda) y$. The $\lambda$-function is defined by $$\lambda (x) := \sup \mathcal{S}(x),$$ where $\displaystyle \mathcal{S}(x):=\{\lambda : (e, y, \lambda) \hbox{ is a triplet amenable to } x\}.$ The space $X$ satisfies the \emph{$\lambda$-property} if $\lambda (x) >0$, for every $x\in X_1$. The Banach space $X$ has \emph{the uniform $\lambda$-property} when $\inf\{\lambda(x): x \in X_1 \} > 0$. Aron, Lohman and Suárez explored the first properties of the $\lambda$-function and gave explicitly the form of this function for certain classical function and sequence spaces in \cite{AronLohm,AronLohmSu}.\smallskip

In \cite[Question 4.1]{AronLohm}, Aron and Lohman posed the following challenge: ``What spaces of operators have the $\lambda$-property and what does the $\lambda$-function look like for these spaces?''. This question motivated a whole series of papers, in which Brown and Pedersen determined the exact form of the $\lambda$-function for every von Neumann algebra and for every unital C$^*$-algebra (cf. \cite{Ped91}, \cite{BrowPed95} and \cite{BrowPed97}). In their study of the $\lambda$-function, Brown and Pedersen introduce the set $A_q^{-1}$ of \emph{quasi-invertible elements} in a unital C$^*$-algebra $A$, and study the geometric properties of $A_1$ in relation to the set $A_q^{-1}$. The following explicit formulae to compute the distance from an element in $A_1$ to the set of quasi-invertible elements or to $\mathfrak{E} (A_1)$ are established by Brown and Pedersen: $$\hbox{dist} (a,\mathfrak{E} (A_1)) =\left\{\begin{array}{lc}
                                                                 \max \left\{ 1- m_q (a) , \|a\|-1\right\}, & \hbox{ if } a\in A_q^{-1}; \\
                                                                 \ & \ \\
                                                                 \max \left\{ 1+ \alpha_q (a) , \|a\|-1\right\}, & \hbox{ if } a\notin A_q^{-1},
                                                               \end{array}
 \right.$$ where $\alpha_q (a) = \hbox{dist} (a,A_q^{-1})$ and $m_q (a)=\hbox{dist} (a,A\backslash A_q^{-1})$ (cf. \cite[Theorem 2.3]{BrowPed97}). The $\lambda$-function is given by $$\lambda (a) = \left\{\begin{array}{lc}
                     \frac{1+m_q (a)}{2}, & \hbox{ if } a\in A_1\cap A_q^{-1}; \\
                     \ & \ \\
                     \frac12 (1-\alpha_q (a)), & \hbox{ if } a\in A_1\backslash A_q^{-1},
                   \end{array}
 \right.$$ (cf. \cite[Theorem 3.7]{BrowPed97}). Furthermore, every von Neumann algebra (i.e. a C$^*$-algebra which is also a dual Banach space) satisfies the uniform $\lambda$-property, actually the expression $\lambda(a) =\frac{1+m_q (a)}{2}$ holds for every element $a$ in the closed unit ball of a von Neumann algebra (cf. \cite[Theorem 4.2]{Ped91}).\smallskip

There exists a class of complex Banach spaces defined by certain holomorphic properties of their open unit balls, we refer to the class of JB$^*$-triples. Harris shows in \cite{Harris74} that the open unit ball of every C$^*$-algebra $A$ is a \emph{bounded symmetric domain} and the same conclusion holds for the open unit ball of every closed linear subspace $U \subseteq A$ invariant under the Jordan triple product \begin{equation}\label{eq triple product C*} \{x,y,z\}: =\frac12 (xy^*z+zy^* x).
 \end{equation}In \cite{Ka83}, Kaup introduces the concept of a JB$^*$-triple and shows that every bounded symmetric domain in a complex Banach space is biholomorphically equivalent to the open unit ball of a JB$^*$-triple and in this way, the category of all bounded symmetric domains with base point is equivalent to the category of JB$^*$-triples. Actually every C$^*$-algebra is a JB$^*$-triple with respect to \eqref{eq triple product C*}, however, the class of JB$^*$-triples is strictly wider than the class of C$^*$-algebras (see next section for definitions and examples).\smallskip

For each complex Banach space $E$ in the class of JB$^*$-triples, the open unit ball of $E$ enjoys similar geometric properties to those exhibited by the closed unit ball of a C$^*$-algebra. Many geometric properties studied in the setting of C$^*$-algebras have been studied in the wider class of JB$^*$-triples. For example, in recent papers, the first, third and fourth authors of this note extend the notion of quasi-invertible elements from the setting of C$^*$-algebras to the wider class of JB$^*$-triples introducing the concept of \emph{Brown-Pedersen quasi-invertible} elements (see \cite{TahPhD}, \cite{TahSiddJam2013}, and \cite{TahSiddJam2014}). Once the class $E_{q}^{-1}$ of Brown-Pedersen quasi-invertible elements in a JB$^*$-triple $E$ has been introduced, the following question seems the natural problem to be studied:

\begin{problem}\label{q lambda in JB*-triples} What JB$^*$-triples have the $\lambda$-property and what does the $\lambda$-function look like in the case of a JB$^*$-triple?
\end{problem}

Only partial answers to the above problem are known. Accordingly to the terminology employed by Brown and Pedersen, for each element $x$ in a JB$^*$-triple $E$, the symbol $\alpha_q (x)$ will denote the distance from $x$ to the set $E_{q}^{-1}$ of Brown-Pedersen quasi-invertible elements in $E$, that is, $\alpha_q (x) = \hbox{dist} (x,E_q^{-1}).$ The know estimates for the $\lambda$-function in the setting of JB$^*$-triples are the following: for each (complete tripotent) $v\in \mathfrak{E} (E_1),$ and each element $x$ in the closed unit ball of the Peirce-2 subspace $E_2(v)$ which is not Brown-Pedersen quasi-invertible in $E$ we have: \begin{equation}\label{ineq lambda function triples JamTahSidd} \lambda (x) \leq \frac12 (1-\alpha_q (x));
\end{equation}
consequently, $\lambda (x) =0$ whenever $\alpha_q (x) = 1$ (cf. \cite[Theorem 3.7]{TahSiddJam2014}). \smallskip

In this paper we continue with the study of the $\lambda$ function in the general setting of JB$^*$-triples. In section 2 we introduce the basic facts and definitions needed in the paper, and we revisit the concept of Brown-Pedersen quasi-invertibility by finding new characterizations of this notion in terms of the triple spectrum and the orthogonal complement of an element.\smallskip

We begin section 3 proving that, for each element $x$ in a JB$^*$-triple $E$, the square-root of the quadratic conorm, $\gamma^{q} (x)$, introduced in \cite{BurKaMoPeRa}, measures the distance from $x$ to the set $E\backslash E_q^{-1}$ (see Theorem \ref{t mq distance}), where by convention $\gamma^{q} (x)=0$ for every $x\in E\backslash E_q^{-1}$. It is established that for every Brown-Pedersen quasi-invertible element $a$ in $E$ we have $$\hbox{dist} (a, \mathfrak{E} (E_1)) = \max \left\{ 1- m_q (a) , \|a\|-1\right\}$$ (see Proposition \ref{p distance to the extremals from a BP qinv element}). This formula is complemented with Theorem \ref{t lambda function 1 from a BPqinv} where we prove that $\lambda (a) = \frac{1+m_q (a)}{2},$ for every Brown-Pedersen quasi-invertible element $a$ in $E_1$.\smallskip

For elements in the closed unit ball of a JB$^*$-triple which are not Brown-Pedersen quasi-invertible, we improve the estimates in \eqref{ineq lambda function triples JamTahSidd} (see \cite{TahSiddJam2014}) by proving that for every JB$^*$-triple  $E$ with $\mathfrak{E} (E_1)\neq  \emptyset,$ the inequalities $$1+\|a\|\geq \hbox{dist} (a, \mathfrak{E} (E_1)) \geq \max \left\{ 1+ \alpha_q (a) ,  \|a\|-1  \right\},$$ hold for every $a$ in $E\backslash E_q^{-1}$ (Theorem \ref{t ineq distance to the extrem from a non BP qinv}). Consequently, the inequality $$\lambda (a)\leq \frac12 (1-\alpha_q (a)),$$ holds for every $a\in E_1\backslash E_q^{-1}$ without assuming that $a$ lies in the Peirce-2 subspace associated with a complete tripotent $v$ in $E$ (see Corollary \ref{c lambda function estimates from a non BPqinv}).\smallskip

A JBW$^*$-triple is a JB$^*$-triple which is also a dual Banach space. In the setting of JBW$^*$-triples play an analogue role to that played by von Neumann algebras in the class of C$^*$-algebras. In section 4 we prove that every JBW$^*$-triple satisfies the uniform $\lambda$-property (see Corollary \ref{c lambda property for JBW}), a result which extends \cite[Theorem 4.2]{Ped91} to the context of JBW$^*$-triples. This result will follow from Theorem \ref{t lambda function in JBW}, where it is established that for every JBW$^*$-triple $W$ the $\lambda$-function on $W_1$ is given by the expression: $$\lambda (a) = \left\{\begin{array}{lc}
                     \frac{1+m_q (a)}{2}, & \hbox{ if } a\in W_1\cap W_q^{-1} \\
                     \ & \ \\
                     \frac12 (1-\alpha_q (a))=\frac12, & \hbox{ if } a\in W_1\backslash W_q^{-1}.
                   \end{array}
 \right.$$

The paper finishes with a result establishing that, for every element $a$ in the closed unit ball of a JB$^*$-triple $E$ which is not Brown-Pedersen quasi-invertible, if $\mathfrak{E} (E_1)\neq  \emptyset,$ then the distance from $a$ to the latter set is given by the formula $$\hbox{dist} (a, \mathfrak{E} (E_1)) = 1+ \alpha_q (a),$$ (see Theorem \ref{t exact distance to the extrem from a non BP qinv}).

\section{von Neumann regularity and Brown-Pedersen invertibility}\label{sec: von Neumann reg and BP inv}

From a purely algebraic point of view, a \emph{complex Jordan triple system}
is a complex linear space ${E}$ equipped with a triple product
$$\{.,.,.\} : { E} \times { E} \times { E} \to { E}$$ $$(x,y,z)
\mapsto \{x,y,z\}$$ which is bilinear and symmetric in the outer
variables and conjugate linear in the middle one and satisfies the
\emph{Jordan identity}:  $$L(x,y) \J abc = \J
{L(x,y)a}bc - \J a{L(y,x)b}c + \J ab{L(x,y)c},$$ for all
$x,y,a,b,c\in { E},$ where $L(x,y): { E} \to { E}$ is the linear
mapping given by $L(x,y) z = \J xyz$.\smallskip

Given an element $a$ in a complex Jordan triple system $E$, the symbol $Q(a)$ will denote the
conjugate linear operator on $E$ given by $Q(a)(x):=\J axa$. It is known that the fundamental identity
\begin{equation}\label{basic equation}
Q(x) Q(y) Q(x) = Q(Q(x)y), \ \ \end{equation} holds for every $x,y$ in a complex Jordan triple system $E$ (cf. \cite[Lemma 1.2.4]{Chu2012}).\smallskip

The studies on von Neumann regular elements in Jordan triple systems began with the contributions of Loos \cite{Loos75} and Fern\'{a}ndez-L\'{o}pez, Garc\'{\i}a Rus, S\'{a}nchez Campos, and Siles
Molina  \cite{FerGarSanSi}. We recall that an element $a$ in a Jordan triple system $E$ is called \emph{von
Neumann regular} if $a\in Q(a) (E)$ and \emph{strongly von Neumann regular} when
$a\in Q(a)^2 (E)$.\smallskip

Enriching the geometrical structure of a complex Jordan triple system, we find the class complex Banach spaces called JB$^*$-triples, introduced by Kaup to classify bounded symmetric domains in arbitrary complex Banach spaces (cf. \cite{Ka83}). More concretely, a \emph{JB$^*$-triple} is a complex Jordan triple system $E$ which is a Banach space satisfying the additional geometric axioms:
\begin{enumerate}[$(a)$]
\item For each $x\in E$, the map $L(x,x)$ is an hermitian operator with
non-negative spectrum;

\item $\| \J xxx \| = \|x\|^{3}$ for all $x\in { E}$.
\end{enumerate} The basic bibliography on JB$^*$-triples can be found in
\cite{Up,Chu2012}.\smallskip

Examples of JB$^*$-triples include all C$^*$-algebras with the triple product given in \eqref{eq triple product C*}, all JB$^*$-algebras with triple
product $$\J abc := (a \circ b^*) \circ c + (c\circ b^*) \circ a - (a\circ c)
\circ b^*,$$ and the Banach space $L(H,K)$ of all bounded linear operators
between two complex Hilbert spaces $H,K$ with respect to \eqref{eq triple product C*}.\smallskip

A JBW$^*$-triple is a JB$^*$-triple which is also a dual Banach space
(with a unique isometric predual  \cite{BarTi}). The triple product of every JBW$^*$-triple
is separately weak$^*$ continuous (cf. \cite{BarTi}),
and the second dual, $E^{**}$, of a JB$^*$-triple $E$ is a JBW$^*$-triple (cf. \cite{Di86b}).\smallskip

In element $a$ in a JB$^*$-triple $E$ is von Neumann regular, if and only if, it is
strongly von Neumann regular if, and only if, there exists $b\in E$ such
that $Q(a) (b) =a,$ $Q(b) (a) =b$ and $[Q(a),Q(b)]:=Q(a)\,Q(b) - Q(b)\, Q(a)=0$ (cf. \cite[Theorem 1]{FerGarSanSi} and
\cite[Lemma 4.1]{Ka96}). Though for a von Neumann regular element $a$ in a JB$^*$-triple $E$, there exist many elements $c$ in $E$ such that $Q(a)(c) =a$,
there exists a unique element $b\in E$  satisfying $Q(a) (b) =a,$ $Q(b) (a) =b$ and $[Q(a),Q(b)]:=Q(a)\,Q(b) - Q(b)\, Q(a)=0$, this unique element $b$ is called the \emph{generalized inverse} of $a$ in $E$ and it is denoted by $a^{\dag}$.\smallskip

The simplest examples of von Neumann regular elements, probably, are tripotents.
We recall that an element $e$ in a JB$^*$-triple $E$
is called \emph{tripotent} when $\J eee =e$. Each tripotent $e$
in $E$ induces a decomposition of $E$ (called the \emph{Peirce decomposition}) in the form
$$E= E_{2} (e) \oplus E_{1} (e) \oplus E_0 (e),$$ where for
$i=0,1,2,$ $E_i (e)$ is the $\frac{i}{2}$ eigenspace of $L(e,e)$.
The Peirce rules affirm that $\J {E_{i}(e)}{E_{j} (e)}{E_{k} (e)}$ is
contained in $E_{i-j+k} (e)$ if $i-j+k \in \{ 0,1,2\}$ and is zero
otherwise. In addition,
$$\J {E_{2} (e)}{E_{0}(e)}{E} = \J {E_{0} (e)}{E_{2}(e)}{E} =0.$$
The projection $P_{k_{}}(e)$ of $E$ onto $E_{k} (e)$ is called the Peirce
$k$-projection. It is known that Peirce projections are contractive (cf. \cite{FriRu85}) and satisfy that $P_{2}(e) = Q(e)^2,$
$P_{1}(e) =2(L(e,e)-Q(e)^2),$ and $P_{0}(e) =Id_E - 2 L(e,e) + Q(e)^2.$ A tripotent $e$ in $E$
is said to be \emph{unitary} if $L(e,e)$ coincides with the identity map on $E,$ that is, $E_2 (e) = E$.
We shall say that $e$ is \emph{complete} when $E_0 (e) =\{0\}$.\smallskip

The Peirce space $E_2 (e)$ is a unital JB$^*$-algebra with unit $e$,
product $x\circ_e y := \J xey$ and involution $x^{*_e} := \J
exe$, respectively. Furthermore, the triple product in $E_2 (e)$ is given by $$\{ a,b,c\} = (a \circ_e b^{*_e}) \circ_e c + (c\circ_e b^{*_e}) \circ_e a - (a\circ_e c) \circ_e b^{*_e} \ (a,b,c\in E_2 (e)).$$

When a C$^*$-algebra $A$ is regarded as a JB$^*$-triple with the product given in \eqref{eq triple product C*}, tripotent elements in $A$ are precisely partial isometries of $A$. A JB$^*$-triple might not contain a single tripotent element (consider, for example, $C_0(0,1]$ the C$^*$-algebra of all complex-valued continuous functions on $[0,1]$ vanishing at $0$). However, since the complete tripotents of a JB$^*$-triple $E$ coincide with the complex and the real extreme points of its closed unit ball (cf. \cite[Lemma 4.1]{BraKaUp78} and \cite[Proposition 3.5]{KaUp77} or \cite[Theorem 3.2.3]{Chu2012}), every JBW$^*$-triple is full of complete tripotents.\smallskip

As shown by Kaup in \cite{Ka96}, the triple spectrum is one of the most appropriate tools to study and determine von Neumann regular elements.  The \emph{triple spectrum} of an element $a$ in a JB$^*$-triple $E$
is the set $$\hbox{Sp}(a) := \Big\{ t\in \CC : a\notin (L(a,a)-t^2
Id_{E}) (E) \Big\}.$$ The extended spectrum of $a$ is the set
$Sp^{'}(a) := Sp(a) \cup \{0\}$. As usually, the
smallest closed complex subtriple of $E$ containing $a$ will be
denoted by $E_a$. The set
$$\Sigma (a) := \Big\{ s\in \CC : (L(a,a)-s Id_{E})|_{E_a} \hbox{ is not
invertible in } L(E_a) \Big\}$$ stands for the usual spectrum of
the restricted operator $L(a,a)|_{E_a}$ in $L(E_a)$. Following standard notation, we assume
that $\Sigma (a) = \emptyset$ whenever $a=0$ (this is actually an
equivalence, compare \cite[Lemma 3.2]{Ka96}). The following properties were established in \cite{Ka96}.
\begin{enumerate}[{\rm $(\Sigma.i)$}] \item $\Sigma (a)$ is a compact
subset of $\RR$ with $\Sigma  (a)\geq 0$ and the origin cannot be
an isolated point of $\Sigma (a)$. The origin cannot be an isolated point of $\hbox{Sp}(a)$ and $Sp(a) =
-Sp(a)$.
\item $\hbox{Sp}(a) = \{ t \in \CC : t^2 \in \Sigma (a)\}$ and
$Sp(a)\neq \emptyset,$ whenever $a\neq 0$.
\item $S_a:=\hbox{Sp}(a)\cap [0,\infty)$ is a compact subset of $\RR,$ $\|a\|\in S_a \subseteq [0,\|a\|],$ and there exists a unique
triple isomorphism $\Psi : E_a \to C_0 (S_a\cup \{0\})$ such
that $\Psi (a) (s) = s$ for every $s\in S_a ,$ where $C_0(S_a\cup \{0\})$ denotes the space of all complex-valued, continuous functions on $S_a\cup \{0\}$ vanishing at zero. If $0\in S_a$ then it is not isolated in $S_a$.
\item The spectrum $\hbox{Sp}(a)$ does not change when computed
with respect to any closed complex subtriple $F\subseteq E $
containing $a$.
\item The element $a$ is von Neumann regular if and only if $0\notin Sp(a).$
\end{enumerate}

The basic properties of the triple spectrum lead us to the \emph{continuous triple functional calculus}. Given an element $a$ in a JB$^*$-triple $E$
and a function $f\in C_{0}(S_{a}\cup \{0\})$, $f_t (a)$ will denote the unique element in $E_a$ which is mapped to $f$ when $E_a$ is identified as JB$^*$-triple with $C_0 (S_a\cup \{0\})$. Consequently, for each natural $n$, the element
$a^{[\frac{1}{2n-1}]}$ coincides with $f_t (a)$, where
$f(\lambda):= \lambda^{\frac{1}{2n-1}}$. When $a$ is an
element in a JBW$^*$-triple $W$, the sequence $(a^{[\frac{1}{2n-1}]})$ converges in the
weak$^*$-topology of $W$ to a tripotent, denoted by $r(a)$,
and called the \emph{range tripotent} of $a$. The tripotent $r(a)$
is the smallest tripotent $e\in W$ satisfying that $a$ is
positive in the JBW$^*$-algebra $W_{2} (e)$ (see, for example, \cite[comments before Lemma 3.1]{EdRu88} or \cite[\S 2]{BuChuZa}).\smallskip

We shall habitually regard a Banach space $X$ as being contained in its bidual,
$X^{**}$, and we identify the weak$^*$-closure, in $X^{**}$, of a
closed subspace $Y$ of $X$ with $Y^{**}$. For an element $a$ in a JB$^*$-triple $E$, the range tripotent $r(a)$ is defined in $E^{**}$.
Having this in mind, the range tripotent of an element $a$ in a JB$^*$-triple is the element in $E_a^{**} \equiv \left(C_0(S_a\cup \{0\})\right)^{**}$ corresponding with the characteristic function of the set $S_a$.\smallskip

We recall that an element $a$ in a unital Jordan Banach algebra $J$ is called invertible whenever there exists $b\in J$ satisfying $a \circ b = 1$ and $a^2 \circ b = a.$ The element $b$ is unique and it will be denoted by $a^{-1}$. The set $J^{-1}= \hbox{inv}(J)$ of all invertible elements in $J$ is open in the norm topology and $a\in J^{-1}$ whenever $\|a-1\|<1$. It is well known that $a$ is invertible if, and only if, the mapping $x\mapsto U_a (x):= 2 (a\circ x)\circ a - a^2 \circ x$ is invertible, and in that case $U_a^{-1}= U_{a^{-1}}$ (see, for example \cite[page 107]{Chu2012}).\smallskip

The reduced minimum modulus was introduced in \cite{BurKaMoPeRa} to introduce an consider the quadratic conorm of an element in a JB$^*$-triple.
The \emph{reduced minimum modulus} of a non-zero bounded linear or conjugate linear
operator $T$ between two normed spaces $X$ and $Y$ is defined by
\begin{equation}\label{red min mod} \gamma (T) := \inf \{ \|T(x)\|
: \hbox{dist}(x,\ker(T)) \geq 1 \}.\end{equation} Following  \cite{Ka}, we set
$\gamma (0) = \infty$ (reader should be awarded that in
\cite{Apos} $\gamma (0) = 0$). When $X$ and $Y$ are Banach spaces, we have
$$\gamma (T) >0 \Leftrightarrow T(X) \hbox{ is norm closed},$$ (cf. \cite[Theorem IV.5.2]{Ka}). The quadratic-conorm of an element $a$ in a JB$^*$-triple
$E$ is defined as the reduced minimum modulus of $Q(a)$ and it will
be denoted by $\gamma^{q} (a),$ that is, $\gamma^{q} (a) = \gamma (Q(a)).$ The main results in \cite{BurKaMoPeRa} show, among many other things, that:
\begin{enumerate}[{\rm $(\Sigma.vi)$}] \item An element $a$ is von Neumann regular if, and only if, $Q(a)$ has norm-closed image if, and only if, the range tripotent $r(a)$ of $a$ lies in $E$ and $a$ is positive and invertible element of the JB$^*$-algebra $E_2 (r(a))$. Furthermore, when $a$ is von Neumann regular we have: $$Q(a) Q(a^{\dag})= P_2(r(a)) = Q(a^{\dag}) Q(a)$$ and $$L(a,a^{\dag})= L(a^{\dag},a)= L(r(a),r(a))$$ (cf. \cite[comments after Lemma 3.2]{Ka2001} or \cite[page 192]{BurKaMoPeRa}).
\item For each element $a$ in $E$, $\gamma^{q} (a) = \inf\{ \Sigma (a)\} = \inf\{ t^2 : t\in\hbox{Sp}(a)\}.$
\end{enumerate}

Let us recall that the \emph{Bergmann operator} associated with a couple of elements $x,y$ in a JB$^*$-triple
$E$ is the mapping $B(x,y): E\times E \to E$ defined by $B(x,y)= Id-2L(x,y)+Q(x)Q(y)$ (compare \cite{Loos} or
\cite[page 305]{Up}).\smallskip

Inspired by the definition of \emph{quasi-invertible} elements in a C$^*$-algebra developed by Brown and Pedersen in \cite{BrowPed95,BrowPed97}, Tahlawi, Siddiqui, and Jamjoom introduced and developed, in \cite{TahPhD,TahSiddJam2013} and \cite{TahSiddJam2014}, the notion of \emph{Brown-Pedersen quasi-invertible} elements in a JB$^*$-triple $E$. An element $a$ in $E$ is Brown-Pedersen quasi-invertible (BP-quasi-invertible for short) if there exists $b\in E$ such that $B(a,b)=0$. It was established in \cite{TahPhD,TahSiddJam2013} that an element $a$ in $E$ is BP-quasi-invertible if and only if one of the following equivalent statements holds: \begin{enumerate}[$(a)$] \item $a$ is von Neumann regular and its range tripotent $r(a)$ is an extreme point of the closed unit ball of $E$ (i.e. $r(a)$ is a complete tripotent of $E$);
\item There exists a complete tripotent $e\in E$ such that $a$ is positive and invertible in the JB$^*$-algebras $E_2 (e)$.
\end{enumerate} The set of all BP-quasi-invertible elements in $E$ is denoted by $E_q^{-1}$. Let us observe that, in principle, the notion of invertibility makes no sense in a general JB$^*$-triple. By \cite[Theorem 8]{TahSiddJam2013}, $E_q^{-1}$ is open in the norm topology (the reason being that, for each complete tripotent $e$, the set of invertible elements in the JB$^*$-algebra $E_2(e)$ is open and the Peirce projections are contractive).\smallskip

Let us observe that when a C$^*$-algebra $A$ is regarded as a JB$^*$-triple with product given by \eqref{eq triple product C*}, the BP-quasi-invertible elements in $A$, as JB$^*$-triple, are exactly the quasi-invertible elements of $A$ in the terminology of Brown and Pedersen in \cite{BrowPed95,BrowPed97}.\smallskip

We shall also need a characterization of BP-quasi-invertible elements in terms of the orthogonal complement. First, we recall that elements
$a,b$ in a JB$^*$-triple $E$ are said to be \emph{orthogonal} (denoted by $a \perp b$) when $L(a,b)=0.$ By \cite[Lemma 1]{BurFerGarMarPe} we know that
$a\perp b$ if, and only if, one of the following statements holds:
\begin{equation}
\label{ref orthogo}\begin{array}{ccc}
  \J aab =0; & a \perp r(b); & r(a) \perp r(b); \\
  & & \\
  E^{**}_2(r(a)) \perp E^{**}_2(r(b));\ \ \ & E_a \perp E_b;\ \ \  & \J bba=0. \\
\end{array}
\end{equation} For each be a subset $M\subseteq E,$ we write
$M_{_E}^\perp$ for the \emph{(orthogonal) annihilator of $M$} defined by $$
M_{_E}^\perp:=\{ y \in E : y \perp x , \forall x \in M \}.$$ It is known that, for each tripotent $e$ in $E$, $\{e\}^{\perp} = E_{0} (e)$. Furthermore, the identity $\{a\}^{\perp} = \left(E^{**}\right)_0 (r(a)) \cap E$ holds for every $a\in E$ (cf. \cite[Lemma 3.2]{BurGarPe11}). We therefore have:

\begin{lemma}\label{l new characterization BPqinv} Let $a$ be an element in a JB$^*$-triple $E$. Then $a$ is BP-quasi-invertible if, and only if, $a$ is von Neumann regular and $\{a\}^{\perp} = \{0\}$.$\hfill\Box$
\end{lemma}

We initiate the novelties with a series of technical lemmas. 

\begin{lemma}\label{l tech 1} Let $e$ be a complete tripotent in a JB$^*$-triple $E$ and let $z$ be an element in $E$. Suppose $P_2 (e) (z)$ is  invertible in the JB$^*$-algebra $E_2 (e)$. Then $z$ is BP-quasi-invertible.
\end{lemma}

\begin{proof} By hypothesis $z_2 = P_2 (e)(z_2)$ is invertible in the JB$^*$-algebra $E_2(e)$ with inverse $z_2^{-1}$, and since $e$ is complete, $z = z_2 + z_1$ where $z_1=P_1 (e) (z).$ Let us observe that $z_2$ is von Neumann regular in $E$ and $z_2^{\dag} = z_2^{-1}.$\smallskip

We claim that the invertibility of $z_2$ in $E_2 (e)$ also implies that $r(z_2)\in E_2(z_2)$ is a unitary tripotent in the JB$^*$-triple $E_2 (e).$ Indeed, since for each $x\in E_2 (e),$ $$x= P_2 (r(z_2)) (x)= Q(z_2) Q(z_2^{-1}) (x) = U_{z_2} U_{z_2^{-1}}(x),$$ we deduce that $P_2 (r(z_2))|_{E_2(e)}= Id_{E_{2} (e)},$ proving the claim.\smallskip

Clearly, $E_2 (e)= E_2 (r(z_2))$. Given $x\in E,$ the condition $$\{r(z_2), x, r(z_2)\}=0$$ implies $0= Q(r(z_2))^2 (x) = P_2 (r(z_2))(x) = P_2 (e) (x)$, and hence $x=P_1 (e) (x)$ lies in $E_1(e)$. Thus,  $E_1 (r(z_2)) \oplus E_0 (r(z_2)) \subseteq E_1 (e).$ Taking $x\in E_0 (r(z_2)),$ having in mind that $e\in E_2 (r(z_2))$, it follows from Peirce arithmetic that $\{e,e,x\}=0,$ which shows that $E_0 (r(z_2)) \subseteq E_0 (e)= \{0\}.$ Therefore, $r(z_2)$ is a complete tripotent in $E$ and $E_j (r(z_2))= E_j (e),$ for every $j=0,1,2.$\smallskip

Now, by Peirce arithmetic we have: $$Q(z) (z_2^{\dag}) = Q(z_2) (z_2^{\dag}) + 2 Q(z_2,z_1)  (z_2^{\dag}) + Q(z_1) (z_2^{\dag})= z_2 + 2 L(z_2,z_2^{\dag})  (z_1) +0$$
$$= z_2 + 2 L(r(z_2),r(z_2)) (z_1)= z_2 + 2 L(e,e) (z_1)= z_2 + z_1 =z,$$ and
$$Q(z_2^{\dag}) (z) = Q(z_2^{\dag}) (z_2) + Q(z_2^{\dag}) (z_1) = z_2^{\dag}.$$ This shows that $z$ is von Neumann regular. Take $a\in \{z\}^{\perp}.$ Since $$0=\{z_2^{-1}, z, a\}= \{z_2^{-1}, z_2, a\} + \{z_2^{-1}, z_1, a\} $$ $$= P_2 (e) (a) + \frac12 P_1 (e) (a) +  \{z_2^{-1}, z_1, P_2 (e) a\} + \{z_2^{-1}, z_1, P_1 (e) (a)\} $$ $$=\hbox{(by Peirce arithmetic)} = P_2 (e) (a) + \frac12 P_1 (e) (a) + \{z_2^{-1}, z_1, P_1 (e) (a)\},$$ which shows that $P_1 (e) (a)=0$, $P_2 (e) (a)=0$, and hence $a=0$. Lemma \ref{l new characterization BPqinv} concludes the proof.
\end{proof}

\begin{remark}\label{r range tripotent of an invertible element in Peirce 2 is unitary}{\rm We would like to isolate the following fact, which has been established in the proof of Lemma \ref{l tech 1} above: For each invertible element $b$ in a unital JB$^*$-algebra, $J$, its range tripotent $r(b)$ is a unitary element belonging to $J$.
}\end{remark}

\begin{corollary}\label{c distance to an extreme} Let $e$ be a complete tripotent in a JB$^*$-triple $E$. Suppose $a$ is an element in $E$ satisfying $\|a-e\|<1$, then $a$ is BP-quasi-invertible.
\end{corollary}

\begin{proof} Having in mind that $P_2 (e)$ is a contractive projection we get $$\|P_2 (e) (a)-e\| = \|P_2 (e) (a-e)\| \leq \|a-e\|<1.$$ Since $E_2 (e)$ is a unital  JB$^*$-algebra with unit $e$, it follows that $P_2 (e) (a)$ is an invertible element in $E_2 (e)$. The conclusion of the corollary follows from Lemma \ref{l tech 1}.\end{proof}

Let $u,v$ be tripotents in a JB$^*$-triple $E$. We recall \cite[\S 5]{Loos} that $u\leq v$ if $v-u$ is a tripotent with $u\perp v-u$. It is known that $u\leq v$ if and only if $P_2 (u) (v) =u$, or equivalently, $L(u,u) (v) = u$ (cf. \cite[Lemma 1.6 and subsequent remarks]{FriRu85}). In particular, $u\leq v$ if and only if $u$ is a projection in the JB$^*$-algebra $E_2 (v).$  Let us observe that the condition $u\geq v$ implies $L(v,v) (u) = u$. However the condition $L(v,v) (u) =u$ need not imply, in general, the inequality $v\geq u$ (cf. Remark \ref{r unitaries}).\smallskip

The following technical lemma will be repeatedly used later.

\begin{lemma}\label{l unitary in Peirce-2 of a complete tripotent is complete} Let $e$ be a complete tripotent in a JB$^*$-triple $E$. Suppose $u$ is a tripotent in $E_2(e)$ satisfying that $L(u,u) (e) =e$. Then $u$ is a complete tripotent of $E.$
\end{lemma}

\begin{proof} Since $L(u,u)e =e$, we deduce that $e\in E_2 (u)$. By Peirce arithmetic, for each $x\in E,$ $$Q(e) (x) = Q(e) P_2(u)(x) \in E_2 (u),$$ which implies $E_2 (e) = P_2 (e) (E) = Q(e)^2 (E)\subseteq E_2 (u).$ Since, we also have $L(e,e) (u)=u$, we get $E_2 (e) = E_2 (u).$ Therefore, the mapping $T= Q(u)|_{E_2(e)}:E_2 (e) \to E_2(2)$ satisfies that $T^2 = P_2 (u)|_{E_2(e)} = P_2 (e)|_{E_{2} (e)}$ is the identity on $E_2 (e)$.\smallskip

Since the triple product of $E_2 (e)$ is given by $\{ a,b,c\} = (a \circ_e b^{*_e}) \circ_e c + (c\circ_e b^{*_e}) \circ_e a - (a\circ_e c) \circ_e b^{*_e}$ ($a,b,c\in E_2 (e)$), we can easily see that $T(x) = U_{u} (x^*)$ and hence $U_{u}$ is an invertible operator in $L(E_2(e))$. We have therefore proved that $u$ is an invertible element in $E_2(e)$. Lemma \ref{l tech 1} gives the desired statement.
\end{proof}

Lemma 4 in \cite{Sidd2007} proves that for every complete tripotent $e$ in a JB$^*$-triple $E$, every unitary element in the JB$^*$-algebra $E_2(e)$ is an extreme point of the closed unit ball of $E$ (i.e. a complete tripotent of $E$). This statement follows as a direct consequence of the above Lemma \ref{l unitary in Peirce-2 of a complete tripotent is complete}. Concretely, let $u$ be a unitary element in the JB$^*$-algebra $E_2 (e)$ (i.e. $u$ is invertible in $E_2 (e)$ with $u^{-1} = u^{*_e}$). Since the triple product on $E_2 (e)$ is given by $$\{ a,b,c\} = (a \circ_e b^{*_e}) \circ_e c + (c\circ_e b^{*_e}) \circ_e a - (a\circ_e c) \circ_e b^{*_e}$$ ($a,b,c\in E_2 (e)$), we can easily see that $\{ u, u, e \} = (u \circ_e u^{*_e}) \circ_e e + (e\circ_e u^{*_e}) \circ_e u - (u\circ_e e) \circ_e u^{*_e}= e$, and Lemma \ref{l unitary in Peirce-2 of a complete tripotent is complete} gives the statement.\smallskip

The following remark clarifies the connections between Lemma \ref{l tech 1}, Corollary \ref{c distance to an extreme}, Lemma \ref{l unitary in Peirce-2 of a complete tripotent is complete}, and \cite[Lemma 4]{Sidd2007}.

\begin{remark}\label{r unitaries} Let $e$ be a complete tripotent in a JB$^*$-triple $E$ and let $v$ be a tripotent in $E_2 (e)$. Then the following statements are equivalent.\begin{enumerate}[$(a)$]\item $v$ is invertible in the JB$^*$-algebra $E_2 (e)$; \item $v$ is a unitary element in the JB$^*$-algebra $E_2 (e)$;
\item $v$ is a unitary element in the JB$^*$-triple $E_2 (e)$;
\item $L(v,v) (e) =e.$
\end{enumerate}

The implication $(a)\Rightarrow (b)$ is established in Remark \ref{r range tripotent of an invertible element in Peirce 2 is unitary}. The implication $(c)\Rightarrow (d)$ is clear. To see $(b)\Rightarrow (c)$, we recall that the triple product in $E_2 (e)$ is given by $$\{ a,b,c\} = (a \circ_e b^{*_e}) \circ_e c + (c\circ_e b^{*_e}) \circ_e a - (a\circ_e c) \circ_e b^{*_e} \ (a,b,c\in E_2 (e)).$$ Since for each $a\in E_2 (e)$, we have $U_v (a^{*_e}) = Q(v) (a)$ (where $U_b (c) := 2 (b \circ_e c^{*_e}) \circ_e b - (b\circ_e b) \circ_e c^{*_e},$ for all  $b,c\in E_2 (e)$), we can deduce that $$P_2 (v) (a) = Q(v)^2 (a) = U_v (U_v (a^{*_e})^{*_e}) =  U_v U_{v^{*_e}} (a) = a,$$ for every $a\in E_2 (e),$ which shows that $P_2 (v) |_{E_2 (e)} = Id_{E_2 (e)},$ and hence $v$ is a unitary tripotent in $E_2 (e).$ To prove  $(d)\Rightarrow (a)$, we recall that $L(v,v)(e)=e$ implies that $e\in E_2 (v),$ and hence $E_2 (e) = E_2 (v)$ because $v\in E_2(e)$, which proves $(d)\Rightarrow (c)$. Furthermore, recalling that $Id_{E_2(e)} = P_2 (v)|_{E_2 (e)} = U_v U_{v^{*_e}},$ we obtain $(a)$.\smallskip

Consider now the statements:
\begin{enumerate}[$(e)$]
\item[$(e)$] $v$ is an extreme point of $\left( E_2 (e)\right)_1,$ or equivalently, $v$ is a complete tripotent in $E_2 (e);$
\item[$(f)$] $v$ is a complete tripotent in $E$.
\end{enumerate}

It should be noted that $(e)\nRightarrow (f) \Rightarrow (e)$, while $(f)$ do not necessarily imply any of the above statements $(a)$-$(d).$ We consider, for example, an infinite-dimensional complex Hilbert space $H$, a complete tripotent $e\in L(H)$ such that $e e^*=1$ and $p= e^* e\neq 1$. In this case $ L(H)_2 (e) = L(H) e^* e$. The element $p$ is a complete tripotent in $L(H)_2 (e)$, and since $0\neq 1-p \perp p$ it follows that $p$ is not complete in $L(H)$ (this shows that $(e)\nRightarrow (f)$). To see the second claim, pick a complete partial isometry $v\in L(e^*e (H))$ such that $vv^*\neq e^* e$ and $v^* v= e^*e.$ It is easy to see that $v$ is a complete tripotent in $L(H)_2 (e)$ and $L(v,v) (e) = \frac12 ( v v^* e + e v^* v) = \frac12 ( v v^*  + e )\neq e.$\smallskip

For more information on extreme points and unitary elements in C$^*$-algebras, JB$^*$-triples and JB-algebras the reader is referred to \cite{Kad51}, \cite{AkWe}, \cite{Ro2010}, \cite{LeNgWong} and \cite[\S 2]{FerMarPe}.
\end{remark}

\section{Distance to the extremals and the $\lambda$-function}

In this section we shall give some formulas to compute the distance from an element in a JB$^*$-triple $E$ to the set $\mathfrak{E} (E_1)$ of extreme points of the closed unit ball of $E$. Since, in some cases, $\mathfrak{E} (E_1)$ may be an empty set, we shall assume that $\mathfrak{E} (E_1)\neq  \emptyset.$\smallskip

Let $E$ be a JB$^*$-triple. According to the terminology employed in \cite{BrowPed95,BrowPed97} and \cite{TahSiddJam2013,TahPhD} and \cite{TahSiddJam2014}, we define $\alpha_q : E \to \mathbb{R}_0^+$, by $\alpha_q (x) = \hbox{dist} (x,E_q^{-1}).$ Inspired by the studies of Brown and Pedersen, we also introduce a mapping $m_q : E\to \mathbb{R}_0^+$ defined by $$m_q (x) :=\left\{
                                                                       \begin{array}{lc}
                                                                         0, & \hbox{ if } x\notin E_q^{-1}; \\
                                                                         \ & \ \\
                                                                         \left(\gamma^{q}(x)\right)^{\frac12}, & \hbox{ if } x\in E_q^{-1}.\\
                                                                       \end{array}
                                                                     \right.$$ Let us notice that, for each $x\in E_q^{-1}$, $$m_q (x) = \inf\Big\{t : t\in \hbox{Sp}(x)\cap [0,\infty) \Big\}= \max\Big\{\varepsilon>0 : ]-\varepsilon, \varepsilon[\cap \hbox{Sp}(x) = \emptyset\Big\},$$ and $m_q (x)>0$ if, and only if, $x\in E_q (x).$\smallskip

We claim that \begin{equation}\label{eq mq is 1-hom} m_q (\lambda x) = |\lambda| m_q (x),
\end{equation} for every $\lambda\in \mathbb{C}\backslash\{0\},$ $a\in E$. Indeed, since $$\left(\mathbb{C}\backslash\{0\}\right) E_q^{-1}= E_q^{-1} \hbox{ and } \mathbb{C} \left(E\backslash E_q^{-1}\right)= E\backslash E_q^{-1},$$ we may reduce to the case $a\in E_q^{-1}$ (compare $(\Sigma.v)$ and $(\Sigma.iii)$). Since $L(\lambda a, \lambda a) = |\lambda|^2 L(a,a)$, it follows that $\Sigma (\lambda a) = |\lambda|^2 \Sigma (a),$ which gives $m_q (\lambda a) = \inf \{ \sqrt{t} : t\in \Sigma (\lambda a)\} = |\lambda| m_q (a).$ \smallskip

As in the C$^*$-algebra setting, our next result shows that $m_q$ is actually a distance (compare \cite[Proposition 1.5]{BrowPed95} for the result in the setting of C$^*$-algebras).

\begin{theorem}\label{t mq distance} Let $E$ be a JB$^*$-triple, then $$m_q (a) = \hbox{dist} (a,E\backslash E_q^{-1}),$$ for every $a\in E.$ In particular, $m_q (a) = \hbox{dist} (a,E\backslash E_q^{-1})= \left(\gamma^q (a)\right)^{\frac12},$ for every $a\in E_q^{-1}.$
\end{theorem}

\begin{proof} We can assume that $a\in E_q^{-1}$. By $(\Sigma.iii)$ and $(\Sigma.v)$, $S_a:=\hbox{Sp}(a)\cap [0,\infty)$ is a compact subset of $\RR,$ $S_a \subseteq [0,\|a\|],$ $\|a\|=\max \left(S_a\right)$, $0< m_q (a) = \left(\gamma^q (a)\right)^{\frac12} =\min \left(S_a\right)$, and there exists a unique triple isomorphism $\Psi : E_a \to C_0 (S_a\cup \{0\})= C(S_a)$ such that $\Psi (a) (s) = s$ for every $s\in S_a.$ The range tripotent $r(a)$ coincides with the unit element in $C(S_a)$. Clearly, $y_0 = a - m_q (a) r(a)$ lies in $E_a\subseteq E$ and contains zero in its triple spectrum, therefore $y_0\in E\backslash E_q^{-1}$. Since $\|a-y_0\| =\| m_q (a) r(a)\| = m_q (a),$ we get $m_q (a) \geq \hbox{dist} (a,E\backslash E_q^{-1}).$\smallskip

To prove the reverse inequality, we first assume that $\|a\|\leq 1$. Arguing by reduction to the absurd, we suppose that $m_q (a) > \hbox{dist} (a,E\backslash E_q^{-1}),$ then there exists $z\in E\backslash E_q^{-1}$ with $\|a-z\|<m_q(a) = \left(\gamma^q (a)\right)^{\frac12}$. Since $a\in E_q^{-1}$, its range tripotent, $r(a),$ is a complete tripotent in $E$, and $a$ is a positive, invertible element in $E_2 (r(a))$. The contractivity of $P_2 (r(a))$, assures that $$\|a- P_2 (r(a)) (z)\| = \| P_2 (r(a)) (a-z)\| \leq \|a-z\| <m_q (a). $$ Now, we compute the distance $$\|P_2 (r(a)) (z)-r(a)\| \leq \| P_2 (r(a)) (z)-a\| + \| a-r(a)\| $$ $$< m_q (a) + \max\{1-m_q (a), \|a\|-1\} = 1.$$ The general theory of invertible elements in JB$^*$-algebras shows that the element $P_2 (r(a)) (z)$ is invertible in $E_2 (r(a))$, because $r(a)$ is the unit element in the latter JB$^*$-algebra. Lemma \ref{l tech 1} implies that $z\in E_q^{-1}$, which contradicts that $z\in E\backslash E_q^{-1}$. We have therefore proved that $m_q (a) = \hbox{dist} (a,E\backslash E_q^{-1}),$ for every $a\in E_q^{-1}$ with $\|a\|\leq 1$. \smallskip

Finally, given $a\in E_q^{-1},$ we have $$m_q \left(\frac{a}{\|a\|}\right) =  \hbox{dist} \left(\frac{a}{\|a\|},E\backslash E_q^{-1}\right),$$ and $\|a\| m_q \Big(\frac{a}{\|a\|}\Big) = m_q (a)$. Therefore, $$ m_q \left(a\right)=\|a\| m_q \left(\frac{a}{\|a\|}\right) \leq \|a\| \left\|\frac{a}{\|a\|} -c \right\| = \Big\|a - {\|a\|}c \Big\| $$ for every $c\in E\backslash E_q^{-1}$, which shows that $$m_q (a) \leq \hbox{dist} \Big({a},\|a \| \left(E\backslash E_q^{-1}\right) \Big) =  \hbox{dist} (a,E\backslash E_q^{-1}).$$
\end{proof}

It was already noticed in \cite[Lemma 25]{TahSiddJam2013} that $$\alpha_q (\lambda x) = |\lambda| \alpha_q (x); \ \alpha_q (x) \leq \| x\|$$ and $$|\alpha_q (x) -\alpha_q (y)| \leq \|x-y\|,$$ for every $x,y\in E$, $\lambda \in \mathbb{C}$. Theorem \ref{t mq distance} implies that \begin{equation}\label{eq mq Lipschitzian} | m_q (x) -m_q (y)| \leq \|x-y\|,
 \end{equation} for every $x,y\in E$.\smallskip

Our next goal is an extension of \cite[Theorem 2.3]{BrowPed97} to the more general setting of JB$^*$-triples, and determines the distance from a BP-quasi-invertible element in a JB$^*$-triple $E$ to the set of extreme points in $E_1$.

\begin{proposition}\label{p distance to the extremals from a BP qinv element} Let $a$ be a BP-quasi-invertible element in a JB$^*$-triple $E$. Then $$\hbox{dist} (a, \mathfrak{E} (E_1)) = \max \left\{ 1- m_q (a) , \|a\|-1\right\}.$$
\end{proposition}

\begin{proof} Again, by $(\Sigma.iii)$ and $(\Sigma.v)$, the set $S_a:=\hbox{Sp}(a)\cap [0,\infty)$ is a compact subset of $\RR,$ $S_a \subseteq [0,\|a\|],$ $\|a\|=\max \left(S_a\right)$, $0< m_q (a) =\min \left(S_a\right)$, and there exists a unique triple isomorphism $\Psi : E_a \to C(S_a)$ such that $\Psi (a) (s) = s$ for every $s\in S_a,$ and the range tripotent $r(a)$ coincides with the unit element in $C(S_a)$. Since $r(a)\in \mathfrak{E} (E_1)$ and $$\hbox{dist} (a, \mathfrak{E} (E_1))\leq \|a-r(a)\| = \max\left\{1- m_q (a) ,  \|a\|-1 \right\}.$$

Given $e\in \mathfrak{E} (E_1)$, we always have $\|a-e\|\geq  \|a\|-1$. Since $$m_q (a) =  m_q (e-(e-a)) | \geq m_q (e)- \|e-a\| = 1-\|e-a\|,$$ we also have $\hbox{dist} (a, \mathfrak{E} (E_1))\geq \max \left\{ 1- m_q (a) ,  \|a\|-1 \right\}$.
\end{proof}

\begin{corollary}\label{c characterization of extreme points among BP-qinv} Let $E$ be a JB$^*$-triple. Then $$\left\{ a\in E_q^{-1} : \|a\|= m_q (a) =\left(\gamma^q (a)\right)^{\frac12}\right\} = ]0,\infty[ \ \mathfrak{E} (E_1).$$ $\hfill\Box$
\end{corollary}

Our next result is a first estimate for the $\lambda$-function, it can be regarded as an appropriate triple version of \cite[Theorem 3.1]{BrowPed97} and \cite[Lemma 2.4]{Sidd2011NYconvex}.

\begin{theorem}\label{t lambda function 1 from a BPqinv} Let $a$ be a BP-quasi-invertible element in the closed unit ball of a JB$^*$-triple $E$. Then for every $\lambda\in [\frac12, \frac{1+m_q (a)}{2}]$ there exist $e,u$ in $\mathfrak{E} (E_1)$ satisfying $$a = \lambda e + (1-\lambda) u.$$ When $1\geq \lambda > \frac{1+m_q (a)}{2}$ such a convex decomposition cannot be obtained. Consequently, $\lambda (a) = \frac{1+m_q (a)}{2},$ for every $a\in E_q^{-1}\cap E_1.$
\end{theorem}

\begin{proof} The range tripotent $r(a)\in \mathfrak{E} (E_1)$ is the unit element of subtriple $E_a\equiv C(S_a)$, where $S_a:=\hbox{Sp}(a)\cap [0,\infty)$ is a compact subset of $\RR,$ $S_a \subseteq [0,\|a\|],$ $\|a\|=\max \left(S_a\right)$, $0< m_q (a) =\min \left(S_a\right)$ and there exits a triple isomorphism $\Psi: E_a \to C(S_a)$ such that $\Psi (a) (s) = s$ ($s\in S_a$). It is part of the folklore in C$^*$-algebra theory that for every $\lambda\in [\frac12, \frac{1+m_q (a)}{2}]$, the function $\Psi (a): s\mapsto s$ can be written in the form $$\Psi (a) = \lambda v_1 + (1-\lambda) v_2,$$ where $v_1,v_2$ are two unitary elements in $C(S_a)$ (see \cite[Lemma 6]{KadPed85} or \cite[Lemma 2.4]{Sidd2011NYconvex} for a proof in a more general setting). Since $v_1,v_2$ are unitary elements in $E_a \equiv C(S_a)$ and $r(a)$ is an extreme point of the closed unit ball of $E$, the tripotents $e= \Psi^{-1} (v_1)$ and $u= \Psi^{-1} (v_2)$ belong to $\mathfrak{E} (E_1)$ (cf. Lemma \ref{l unitary in Peirce-2 of a complete tripotent is complete}) and $a = \lambda e + (1-\lambda) u.$\smallskip

Given $1\geq \lambda > \frac{1+m_q (a)}{2}$, if we assume that  $a = \lambda e + (1-\lambda) y,$ where $e\in \mathfrak{E} (E_1)$ and $y\in E_1$, we have $$\|x-u\| = (1-\lambda) \|y-e \| \leq 2 (1-\lambda),$$ which shows that $\hbox{dist} (a, \mathfrak{E} (E_1)) \leq 2 (1-\lambda)$. However, by Proposition \ref{p distance to the extremals from a BP qinv element}, $1-m_q(a) = \hbox{dist} (a, \mathfrak{E} (E_1))$, and hence $\lambda \leq \frac{1+m_q (a)}{2},$ which is impossible.
\end{proof}

Our next result was in \cite[Theorem 3.5]{TahSiddJam2014}. We can give now an alternative proof from the above results.

\begin{corollary}\label{c new characterization of BPQI} Let $E$ be a JB$^{*}$-triple. Let $a$ be an element in $E_{1}$. Then $a\in E_{q}^{-1}$ if and only if $a= \alpha v_{1}+(1-\alpha)v_{2}$ for some extreme points $v_{1},v_{2}$ in $\mathcal{E}(E_{1})$ and $0\leq \alpha < \frac{1}{2}$.
\end{corollary}

\begin{proof} ($\Rightarrow$) Since $a \in E_{q}^{-1}\backslash \mathfrak{E} (E_1)$, the distance $m_q$, satisfies $0< m_q (a) < 1$, and hence $(\frac12, \frac{1+m_q (a)}{2}]\neq \emptyset$. Take $\lambda \in (\frac12, \frac{1+m_q (a)}{2}]$. Theorem \ref{t lambda function 1 from a BPqinv} implies the existence of $v_1,v_2$ in $\mathcal{E}(E_{1})$ satisfying $a = \lambda v_2 + (1-\lambda) v_1$. The statement follows for $\alpha = 1-\lambda$.\smallskip

($\Leftarrow$) Note that $\|a - v_{2}\| =\alpha\ \| v_{1} - v_{2}\| < 1$. Corollary \ref{c distance to an extreme} implies that $a\in (\mathcal{J})_{q}^{-1}$.
\end{proof}

In \cite[Theorem 26]{TahSiddJam2013} the authors show that given a complete tripotent $e$ in a JB$^*$-triple $E$ (i.e. $e\in \mathfrak{E} (E_1)$), then for each element $a$ in $E_2 (e)\backslash E_q^{-1}$ we have: $$\hbox{dist} (a, \mathfrak{E} (E_1)) \geq \max \left\{ 1+ \alpha_q (a) ,  \|a\|-1\right\}.$$ Our next result shows that there is no need to assume that the element $a$ lies in the Peirce-2 subspace of a complete tripotent to prove the same inequality.

\begin{theorem}\label{t ineq distance to the extrem from a non BP qinv} Let $E$ be a JB$^*$-triple satisfying $\mathfrak{E} (E_1)\neq  \emptyset.$ Then the inequalities $$1+\|a\|\geq \hbox{dist} (a, \mathfrak{E} (E_1)) \geq \max \left\{ 1+ \alpha_q (a) ,  \|a\|-1  \right\},$$ hold for every $a$ in $E\backslash E_q^{-1}$.
\end{theorem}

\begin{proof} Let us fix $a$ in $E\backslash E_q^{-1}$. Clearly, for each $e\in \mathfrak{E} (E_1)$, $\|a-e\| \geq \left| \|a\|-1 \right|$, and hence $$\hbox{dist} (a, \mathfrak{E} (E_1)) \geq \left| \|a\|-1 \right|.$$

Fix an arbitrary $e\in \mathfrak{E} (E_1)$. If $\|a-e\|<\beta$ then $\beta > 1$, otherwise $\|a-e\|<1$ and Corollary \ref{c distance to an extreme} implies that $a\in E_q^{-1}$, which is impossible. Now, the inequality $$m_q ((\beta-1) e +a) = m_q (\beta e + a-e) \geq m_q (\beta e) -\|a-e\| = \beta -\|a-e\| > 0,$$ shows that $(\beta-1) e +a$ lies in $E_q^{-1}$. Then $$\alpha_q (a) \leq \|a - ((\beta-1) e +a)\| = \beta-1.$$ This proves that $$\alpha_q (a) +1 \leq \beta,$$ for every $e\in \mathfrak{E} (E_1)$ and $\beta > \|a-e\|$, witnessing that $\hbox{dist} (a, \mathfrak{E} (E_1)) \geq 1+ \alpha_q (a)$.\end{proof}

\begin{corollary}\label{c lambda function estimates from a non BPqinv} Let $E$ be a JB$^*$-triple satisfying $\mathfrak{E} (E_1)\neq  \emptyset.$ Then $$\lambda (a)\leq \frac12 (1-\alpha_q (a)),$$ for every $a\in E_1\backslash E_q^{-1}.$
\end{corollary}

\begin{proof} Let us fix $a\in E_1 \backslash E_q^{-1}$. By Theorem \ref{t ineq distance to the extrem from a non BP qinv} we have $$\hbox{dist}(a,\mathfrak{E} (E_1)\geq \max \{ \alpha_q (a) +1 , \|a\|-1\}. $$ Thus, if $a$ writes in the form $a = \lambda e + (1-\lambda) y$, where $e\in \mathfrak{E} (E_1)$, $y\in E_1$ and $0\leq\lambda \leq 1$ we have $a-e = (\lambda-1) e+ (1-\lambda) y,$ which gives $$\alpha_q (a) + 1 \leq \hbox{dist}(a, \mathfrak{E} (\mathcal{J})_1) \leq \|a-e\| = |1-\lambda| \|y-e\| \leq 2 (1-\lambda),$$ which proves $\lambda \leq \frac12 (1-\alpha_q (a))$.
\end{proof}

\section{The $\lambda$-function of a JBW$^*$-triple}

We can present now a precise description of the $\lambda$-function in the case of a JBW$^*$-triple. The main goal of this section is to prove that every JBW$^*$-triple satisfies the uniform $\lambda$-property, extending the result established by Pedersen in \cite[Theorem 4.2]{Ped91} in the context of von Neumann algebras.\smallskip

First we observe that whenever we replace JB$^*$-triples with JBW$^*$-triples the $\alpha_q$ function is much more simpler to compute on the closed unit ball.

\begin{proposition}\label{p alpha_q in JBW} Let $W$ be a JBW$^*$-triple. Then, for each $a$ in $W_1$ we have $$\hbox{dist} (a, \mathfrak{E} (W_1)) = 1- m_q (a).$$ In particular $\alpha_q (a) = 0,$ for every $a\in W_1\backslash W_q^{-1}.$
\end{proposition}

\begin{proof} When $a\in W_q^{-1}$ the statement follows from Proposition \ref{p distance to the extremals from a BP qinv element}. Let us assume that $a\notin W_q^{-1}$, then $0$ is not an isolated point in $S_a$ (cf. $(\Sigma.i)$). One more time, we shall identify $W_a$ (the (norm-closed) JB$^*$-subtriple of $W$ generated by $a$) with $C_0 (S_a\cup \{0\})$. Therefore, for each $\delta>0$ the sets $]\delta,\|a\|]\cap S_a$ and $]0, \delta]\cap S_a$ are non-empty. The characteristic function $r_\delta=\chi_{_{]\delta,\|a\|]}}\in \left(W_a\right)^{\sigma(W,W_*)}$ is a range tripotent of an element in $W_a$, and hence $r_{\delta}$ is a tripotent in $W$.\smallskip

By \cite[Lemma 3.12]{Horn87}, there exists $e\in \mathfrak{E} (W_1))$ such that $Q(e) (r_\delta) = r_{\delta},$ that is $e= r_{\delta} + (e-r_{\delta})$ and $ r_{\delta} \perp (e-r_{\delta})$. Since $P_1 (r_{\delta}) (a-e) =0,$ we can write $$a-e = P_2 (r_{\delta}) (a-e) + P_0 (r_{\delta}) (a-e)= P_2 (r_{\delta}) (a-r_{\delta}) + P_0 (r_{\delta}) (a-e).$$ Clearly, $$\|P_2 (r_{\delta}) (a-r_{\delta})\| = \max\{1-\delta, \|a\|-1 \} = 1-\delta,$$ while $\|P_0 (r_{\delta}) (a-e)\|\leq \|P_0 (r_{\delta}) (a)\|+ \|P_0 (r_{\delta}) (e) \|\leq 1+\delta.$ Now, observing that $P_2 (r_{\delta}) (a-r_{\delta}) \perp P_0 (r_{\delta}) (a-e)$, we deduce from \cite[Lemma 1.3$(a)$]{FriRu85} that $$\hbox{dist} (a, \mathfrak{E} (W_1))\leq \|a-e\|\leq \max\{1+\delta,1-\delta \} = 1+\delta.$$ The arbitrariness of $\delta>0$ implies that $\hbox{dist} (a, \mathfrak{E} (W_1))\leq 1$.

Finally, the equality $\hbox{dist} (a, \mathfrak{E} (W_1))= 1$ and the final statement follow from Theorem \ref{t ineq distance to the extrem from a non BP qinv}.
\end{proof}

The detailed description of the $\lambda$-function in the case of a JBW$^*$-triple reads as follows:

\begin{theorem}\label{t lambda function in JBW} Let $W$ be a JBW$^*$-triple. Then the $\lambda$-function on $W_1$ is given by the expression: $$\lambda (a) = \left\{\begin{array}{lc}
                     \frac{1+m_q (a)}{2}, & \hbox{ if } a\in W_1\cap W_q^{-1} \\
                     \ & \ \\
                     \frac12 (1-\alpha_q (a))=\frac12, & \hbox{ if } a\in W_1\backslash W_q^{-1}.
                   \end{array}
 \right.$$
\end{theorem}

\begin{proof} The case $a\in W_1\cap W_q^{-1}$ follows from Theorem \ref{t lambda function 1 from a BPqinv}. Suppose $ a\in W_1\backslash W_q^{-1}.$ Corollary \ref{c lambda function estimates from a non BPqinv} and Proposition \ref{p alpha_q in JBW} imply that $\lambda (a)\leq \frac12 (1-\alpha_q (a)) = \frac12.$

Let $r=r(a)$ denote the range tripotent of $a$ in $W$. Let us observe that, by \cite[Lemma 3.12]{Horn87}, there exists a complete tripotent $e\in \mathfrak{E} (W_1))$ such that $e= r_{\delta} + (e-r_{\delta})$ and $ r_{\delta} \perp (e-r_{\delta})$. This implies that $a$ is a positive element in the closed unit ball of the JBW$^*$-algebra $W_2(e)$. Since $a\notin W_q^{-1}$, $0$ lies in the triple spectrum of $a$ (cf. $(\Sigma.v)$). Furthermore, the triple spectrum of $a$ does not change when computed as an element in $W_2(e)$ (see $(\Sigma.iv)$), thus $a$ is not BP-quasi-invertible in $W_2(e)$. Let $\mathcal{J}_{a,e}$ denote the JBW$^*$-algebra of $W_2(e)$ generated by $e$ and $a$. It is known that $\mathcal{J}_{a,e}$ is isometrically isomorphic, as JBW$^*$-algebra, to a an abelian von Neumann algebra with unit $e$ (compare \cite[Lemma 4.1.11]{HancheStor}). Since, in the terminology of \cite{BrowPed95,BrowPed97}, $a$ neither is quasi-invertible in the abelian von Neumann algebra
$\mathcal{J}_{a,e}$, we deduce, via \cite[Theorem 4.2]{Ped91}, that there exist unitary elements $e_1$ and $e_2$ in $\mathcal{J}_{a,e}$ satisfying $a = \frac12 e_1 + \frac12 e_2$. Since $e\in \mathfrak{E} (W_1)$ is the unit element in $\mathcal{J}_{a,e}$ and $e_1, e_2$ are unitary element in the latter von Neumann algebra, we conclude that $e_1,e_2 \in \mathfrak{E} (W_1)$ (cf. Lemma \ref{l unitary in Peirce-2 of a complete tripotent is complete} or \cite[Lemma 4]{Sidd2007}), which shows that $\frac12 \leq \lambda(a).$
\end{proof}

As in the C$^*$-setting, an element $a$ in the closed unit ball of a JBW$^*$-triple is BP-quasi-invertible if, and only if, $\lambda (a) >\frac12.$\smallskip

\begin{corollary}\label{c lambda property for JBW} Every JBW$^*$-triple satisfies the uniform $\lambda$-property.$\hfill\Box$
\end{corollary}

In \cite[\S 4]{TahSiddJam2014} (see also \cite[\S 5.3]{TahPhD}), the authors introduce the $\Lambda$-condition in the setting of JB$^*$-triples in the following sense: a JB$^*$-triple $E$ satisfies the $\Lambda$-condition if for each complete tripotent $e\in \mathfrak{E} (E)$ and every $a\in \left(E_2 (e)\right)_1\backslash E_q^{-1},$ the condition $\lambda (a)=0$ implies $\alpha_q (a)=1.$ We can affirm now that every JBW$^*$-triple actually satisfies a stronger property, because, by Theorem \ref{t lambda function in JBW} (see also Proposition \ref{p alpha_q in JBW}), the minimum value of the $\lambda$-function on the closed unit ball of a JBW$^*$-triple is $\frac12$ (compare with \cite[Theorem 4.2]{Ped91} for the appropriate result in von Neumann algebras).\smallskip

Our next goal is to complete the statement of Theorem \ref{t ineq distance to the extrem from a non BP qinv} in the case of a general JB$^*$-triple.

\begin{proposition}\label{p toward de Lamba condition for JB} Let $a$ and $b$ be elements in a JB$^*$-triple $E$. Suppose $\|a-b\|<\beta$ and $b\in E_q^{-1}$. Then $a + \beta r(b)\in E_q^{-1}$ and the inequality $$m_q (a+ \beta r(b)) \geq \beta -\|b-a\|,$$ holds. Furthermore, under the above conditions, the element $P_2 (r(b)) (a) +  \beta r(b)$ is invertible in the JB$^*$-algebra $E_2 (r(b))$.
\end{proposition}

\begin{proof} Let us write $a +\beta r(b) = a-b + b+ \beta r(b).$ Considering the JB$^*$-subtriple $E_b$ generated by $b,$ we can easily see that $m_q ( b+ \beta r(b)) = \beta + m_q (b).$ Therefore, by \eqref{eq mq Lipschitzian}, $$m_q (a+ \beta r(b)) \geq m_q ( b+ \beta r(b)) -\|a-b\| = \beta + m_q (b) -\|a-b\| >\beta -\|b-a\|>0,$$ which proves the first statement.\smallskip

Now, set $c = P_2 (r(b)) (a-b)$. Clearly, $\|c\|\leq \|a-b\|<\beta$. We write $$ P_2 (r(b)) (a) +  \beta r(b) = c + P_2 (r(b)) (b) +  \beta r(b) = c + b + \beta r(b). $$ Since $b$ is invertible and positive in the JB$^*$-algebra $E_2 (r(b))$, we deduce that $d = b+\beta r(b)$ is a positive invertible element in $E_2 (r(b))$, with inverse $d^{-1}\in E_2 (r(b))$. It is easy to see that $\|d^{-1}\|^{-1} = \|(b+\beta r(b))^{-1} \|^{-1} \geq \beta +m_q (b) > \beta, $ and hence $$\left\|U_{d^{-\frac12}} (a-b) \right\| \leq \|d^{-\frac12}\|^2 \ \|a-b\| < \frac{1}{\beta} \ \beta = 1,$$ which implies that $r(b) + U_{d^{-\frac12}} (a-b)$ is invertible in the JB$^*$-algebra $E_2 (r(b))$. Finally, the identity $$ P_2 (r(b)) (a) +  \beta r(b) =P_2 (r(b)) (a-b) + P_2 (r(b)) (b) +\beta r(b) $$ $$= U_{d^{\frac12}} \left( U_{d^{-\frac12}} (a-b) + U_{d^{-\frac12}} ( b+ \beta r(b))\right) = U_{d^{\frac12}} \left( U_{d^{-\frac12}} (a-b) +   r(b)\right), $$ gives the final statement and concludes the proof.
\end{proof}

We can now extend Proposition \ref{p alpha_q in JBW} to the setting of JB$^*$-triples.

\begin{theorem}\label{t exact distance to the extrem from a non BP qinv} Let $E$ be a JB$^*$-triple satisfying $\mathfrak{E} (E_1)\neq  \emptyset.$ Then the formula $$\hbox{dist} (a, \mathfrak{E} (E_1)) = 1+ \alpha_q (a) ,$$ holds for every $a$ in $E_1\backslash E_q^{-1}$.
\end{theorem}

\begin{proof} Fix $a$ in $E_1\backslash E_q^{-1}$. Theorem \ref{t ineq distance to the extrem from a non BP qinv} proves that $2\geq \hbox{dist} (a, \mathfrak{E} (E_1)) \geq  1+ \alpha_q (a).$ In particular $0\leq \alpha_q (a) \leq 1.$ When $\alpha_q (a)=1$, we have $ 2\geq \hbox{dist} (a, \mathfrak{E} (E_1)) \geq 1+ \alpha_q (a) =2$, we may therefore assume that $\alpha_q(a)<1$.\smallskip

We shall prove now that for each pair $(\delta,\beta)$ with  $1> \delta>\beta > \alpha_q (a)$ there exists $e\in \mathfrak{E} (E_1)$ with $\|a-e\| < \max\{1+ \beta, 2 \delta \}.$ Indeed, by definition there exists $b\in E_q^{-1}$ such that $\|a-b\|< \beta < \delta.$ By Proposition \ref{p toward de Lamba condition for JB}, the element $z=a + \delta r(b) \in E_q^{-1}$ and $m_q(z) = m_q (a+ \delta r(b)) \geq \delta -\|b-a\|> \delta-\beta.$\smallskip

Clearly, $\|a-z\| = \|\delta r(b)\| = \delta.$ Since $z\in E_q^{-1},$ its range tripotent $e=r(z)\in \mathfrak{E} (E_1).$ It is known that $$\|z-r(z)\| = \max\{1-m_q(z), \|z\|-1 \} < \max\{1-\delta+ \beta, \|a\|+\delta-1 \}$$ Therefore $$\|a -r(z)\| \leq \|a-z\| + \|z- r(z)\| $$ $$< \delta + \max\{1-\delta+ \beta, \|a\|+\delta-1 \} \leq \max\{1+ \beta, 2 \delta \}.$$

This proves that for each pair $(\delta,\beta)$ with  $1> \delta>\beta > \alpha_q (a)$ we have $$\hbox{dist} (a, \mathfrak{E} (E_1)) \leq \max\{1+ \beta, 2 \delta \},$$ letting $\beta,\delta\to \alpha_q (a)$ we get $$\hbox{dist} (a, \mathfrak{E} (E_1)) \leq \max\{1+ \alpha_q (a), 2 \alpha_q (a) \}= 1+ \alpha_q (a),$$ which concludes the proof.
\end{proof}

The set of extreme points of the closed unit ball of a unital C$^*$-algebra is always non-empty. Since every C$^*$-algebra is a JB$^*$-triple, \cite[Theorem 2.3]{BrowPed97} derives as a direct consequence of our Theorem \ref{t exact distance to the extrem from a non BP qinv}. Actually the proof above provides a simpler argument to obtain the result in \cite{BrowPed97}. Let us observe that the introduction of JB$^*$-triple techniques makes the proofs easier because the set of extreme points is not directly linked to the order structure of a C$^*$-algebra.\smallskip

\begin{remark}\label{remark open question} In order to determine the $\lambda$ function on $E_1\backslash E_q^{-1}$, it would be very interesting to know if the distance formula established in Theorem \ref{t exact distance to the extrem from a non BP qinv} can be improved to show that, under the same hypothesis, the equality \begin{equation}\label{eq conjectured distance} \hbox{dist} (a, \mathfrak{E} (E_1)) = \max\{1+ \alpha_q (a) , \|a\|-1\},
\end{equation} holds for every $a$ in $E\backslash E_q^{-1}$.\smallskip

We recall that a JB$^*$-triple $E$ is said to be \emph{commutative} or \emph{abelian} if the identity $$ \J {\J xyz}ab = \J xy{\J zab} = \J x{\J yza}b$$ holds for all
$x,y,z,a,b \in E$, equivalently, $L(a,b) L(c,d) = L(c,d) L(a,b)$, for every $a,b,c,d\in E$. Suppose $E$ is a commutative JB$^*$-triple with $\mathfrak{E} (E_1)\neq  \emptyset.$ It is known (cf. \cite[Theorems 2 and 4]{FriRus82} or \cite[Lemma 6.2]{HoPeRu2013}) that for each $e\in \mathfrak{E} (E_1),$ the JB$^*$-triple $E$ is a commutative C$^*$-algebra with unit $e$, product and involution given by $a\circ_e b := \{ a,e,b\}$ and $a^{\sharp_e} := \{ e,a,e\}$ ($a,b\in E$), respectively, and the same norm. We have already observed that when a C$^*$-algebra $A$ is regarded as a JB$^*$-triple with the triple product given in \eqref{eq triple product C*}, the BP-quasi-invertible elements in $A$, as JB$^*$-triple, are exactly the quasi-invertible elements of the C$^*$-algebra $A$ introduced and studied by Brown and Pedersen in \cite{BrowPed95,BrowPed97}. Since the Banach space underlying $E$ has not been changed, we can deduce from \cite[Theorem 2.3]{BrowPed97} that $$\hbox{dist} (a, \mathfrak{E} (E_1)) = \max\{1+ \alpha_q (a) , \|a\|-1\},$$ for every $a$ in $E\backslash E_q^{-1}$, that is, \eqref{eq conjectured distance} holds for every commutative JB$^*$-triple $E$ with $\mathfrak{E} (E_1)\neq  \emptyset.$ It can be also shown that in this case, $$\lambda (a)\leq \frac12 (1-\alpha_q (a)),$$ for every $a\in E_1\backslash E_q^{-1}.$ Since commutative JB$^*$-triples are also example of function spaces (cf. \cite[\S 1]{Ka83} and \cite{FriRus83}), the last result complements the study developed in \cite[Theorem 1.9]{AronLohm}.
\end{remark}

\bigskip

\end{document}